\documentclass[10pt]{amsart}
\usepackage{amsmath,amssymb}
\usepackage{amsxtra, amsmath}
\usepackage{amssymb, amscd}
\usepackage{graphicx}
\usepackage{url}

\newcommand\CC{\mathbb C}

\newcommand\RR{\mathbb R}

\newcommand\NN{\mathbb N}

\theoremstyle{plain}
\newtheorem{thm}{Theorem}[section]

\newtheorem{prop}[thm]{Proposition}
\newtheorem{cor}[thm]{Corollary}
\newtheorem{conj}[thm]{Conjecture}
\theoremstyle{definition}

\theoremstyle{remark}
\newtheorem{remark}[thm]{Remark}

\title{Matrix Gegenbauer
Polynomials: the $2\times 2$ Fundamental Cases }
\author{In\'es Pacharoni}
\author{Ignacio Zurri\'an}
\address{CIEM-FaMAF, Universidad Nacional de C\'ordoba, 5000 C\'ordoba, Argentina.}
\email{ pacharon@famaf.unc.edu.ar, zurrian@famaf.unc.edu.ar}

\thanks{This paper was partially supported by CONICET, PIP 112-200801-01533 and SeCyT-UNC}

\begin{document}
\begin{abstract}
In this paper, we exhibit explicitly a sequence  of $2\times2$ matrix valued orthogonal polynomials with respect to a weight $W_{p,n}$,
for any pair of real numbers $p$ and $n$ such that $0<p<n$.
 The entries of these polynomiales are expressed in
terms of the Gegenbauer polynomials $C_k^\lambda$. Also the corresponding three-term recursion relations are given and we make some studies of the
algebra  of differential operators associated with the weight $W_{p,n}$.
\end{abstract}

\maketitle

\section{Introduction}
 The theory of matrix valued orthogonal polynomials, without any
 consideration of differential equations, goes back to \cite{K49} and
 \cite{K71}. In \cite{D97}, the study of the matrix valued orthogonal
 polynomials that are eigenfunctions of certain second order symmetric
 differential operators was started. The first explicit examples of such polynomials were given in \cite{GPT01}, \cite{GPT02a}, \cite{G03},  \cite{GPT03} and
 \cite{DG04}. See also \cite{DG05a}, \cite{DG05b},  \cite{CG05}, \cite{CG06},  and the references given there.

On the two dimensional sphere $S^2=\mathrm{SO}(3)/\mathrm{SO}(2)$,
the harmonic analysis with respect to the action of the orthogonal group is contained in the classical theory of the
spherical harmonics. In spherical coordinates, the zonal spherical functions on $S^2$ are the Legendre polynomials.
More generally, in the case of the $n$-dimensional sphere $S^n$
the zonal spherical functions are given  in terms of Gegenbauer (or ultraspherical) polynomials of parameter $(n-1)/2$.

This fruitful connection between orthogonal polynomials and representation theory of compact Lie groups is also established in the matrix case:
the matrix valued spherical functions of any $K$-type are closely related to matrix valued orthogonal polynomials.
In this way, several examples of matrix orthogonal polynomials which are eigenfunctions of a symmetric differential operator have been obtained by focusing on a group representation  approach.
   See for example \cite{GPT02a},  \cite{GPT05}, \cite{PT04}, \cite{PT07b}, \cite{PR08} and more recently    \cite{KPR12} and \cite{PTZ13}.

 The examples of matrix orthogonal polynomials introduced in this paper are  motivated by  the spherical functions of fundamental $K$-types associated with the  $n$-dimensional spheres $S^n\simeq G/K$, where $(G,K)=(\mathrm{SO}(n+1), \mathrm{SO}(n))$. These matrix valued spherical functions were studied in detail in \cite{TZ13B} and \cite{Z13}. The ``group  parameters" of  the fundamental $K$-types are  $p$, $n\in \NN$  such that $0<p<[n/2]$ and they  give rise to $2\times 2$ matrix valued orthogonal polynomials.

 In this paper we go beyond these group parameters and we extend these parameters continuously.
 We would like to remark that the group representation theory is a natural source of examples of matrix valued orthogonal polynomials.
We keep this in mind in spite of the fact that the results obtained in this paper
are self-contained, the proofs are of analytic nature and   they do not depend on any previous results on spherical functions.

Given a weight matrix $W$, it is very natural to study the algebra  $\mathcal D(W)$, of all differential
operators that have a sequence of matrix valued orthogonal polynomials with respect to  $W$ as eigenfunctions, see \eqref{algDW}.
 In the classical cases of Hermite, Laguerre and Jacobi weights,
 the structure of this algebra is well understood: it is a polynomial algebra in a second order differential operator, see \cite{M05}. In particular, it is a commutative algebra.
In the matrix  case, the first attempt to go beyond the issue of the existence of one nontrivial element in $\mathcal D(W)$  and to  study the full algebra is undertaken in \cite{CG06}, with the assistance of  symbolic computation, for a few  weights $W$.
The   first deep study of the  algebra  
$\mathcal D(W)$ can be founded in \cite{T11}, where the author worked out one of the examples introduced in \cite{CG06}.
We refer the reader to \cite{GT07} for basic definitions and main results concerning the algebra $\mathcal D(W)$. The present paper leads to understand completely a second and more promising example of $\mathcal D(W)$  in a forthcoming paper, \cite{Z15}. There are very few examples of non-commutative algebras that arise in a natural setup at the intersection of  harmonic analysis and
algebras. The study of such examples for the algebra $\mathcal D(W)$ considered here is one step in that direction.
++As a consequence of this work, together with F.A. Gr\"unbaum, in \cite{GPZ14} we extend to a matrix setup a result that traces its origin
and its importance to the work of Claude Shannon in lying the mathematical
 foundations of information theory,
and to a remarkable series of papers by D. Slepian, H. Landau and H. Pollak.

To the best of our knowledge,  this is the first example showing in a non-commutative setup that a bispectral property implies that the corresponding global operator of ``time and band limiting" admits a commuting local operator. This is a noncommutative analog of the famous prolate spheroidal wave operator.

\smallskip
Now we discuss briefly the content of the paper.
In Section \ref{sec-MOP} we recall the general notions of matrix valued orthogonal polynomials
and some results from \cite{GT07} about the
algebra $\mathcal D(W)$.

In Section \ref{sec-MOPn},  we introduce our sequence $\{P_w\}_{w\in\NN_0}$ of $2\times 2$ matrix valued
polynomials on $[-1,1] $ whose entries are given in terms of the classical Gegenbauer polynomials,
for  real parameters  $p$ and $n$ such that $0< p< n$, see \eqref{Pwdef}.
 We prove that these polynomials  satisfy $P_wD=\Lambda_wP_w$, where
$D$ is a (right-hand side) hypergeometric differential operator
and the eigenvalue is a diagonal matrix. This differential operator $D$ is symmetric with respect to the matrix weight
$W$ introduced in \eqref{peso-x}.
We use these facts to prove that the polynomials $\{P_w\}_{{w\in\NN_0}}$ are orthogonal with respect to the weight matrix $W=W_{p,n}$
 (Theorem \ref{PwortogonalW}).

We  also connect our  weight matrix $W_{p,n}$ with the weight  considered in   \cite{KRR14}, where the authors give examples of matrix valued Gegenbauer polynomials, extending for an arbitrary parameter $\nu$  the  results in \cite{KPR12} for $\nu=1$. See Remark \ref{comparacionK}.

In Section \ref{sec-ttrr} we prove a three-term recursion relation sa\-tis\-fied by
$\{P_w\}_{w\in\NN_0}$.
Section \ref{sec-dw} is focused on the study of the  algebra $\mathcal D(W)$.
In our  case  $\mathcal D(W)$ is a noncommutative algebra.
We provide a basis $\{D_1,D_2,D_3,D_4, I\}$ of the subspace  of the differential operators in $\mathcal D(W)$ of order at most two.
The differential operators $D_1$ and $D_2$ are symmetric operators, while $D_3$ and $D_4$ are not.
We conjecture that $D_1, D_2, D_3, D_4$ generates the  algebra $\mathcal D(W)$.

\section{Background on matrix valued orthogonal polynomials}\label{sec-MOP}

 Let $W=W(x)$ be a weight matrix of size $N$ on the real line, that is a complex $N\times N$ matrix valued integrable function on the interval $(a,b)$ such that $W(x)$ is positive definite almost everywhere and with finite moments of all orders. Let $\operatorname{Mat}_N(\CC)$ be the algebra of all $N\times N$ complex matrices and let $\operatorname{Mat}_N(\CC)[x]$ be the algebra over $\CC$ of all polynomials in the indeterminate $x$ with coefficients in $\operatorname{Mat}_N(\CC)$. We consider the following Hermitian sesquilinear form in the linear space $\operatorname{Mat}_N(\CC)[x]$
\begin{equation*}
  \langle P,Q \rangle =  \langle P,Q \rangle_W = \int_a^b P(x) W(x) Q(x)^*\,dx.
\end{equation*}
The following properties are satisfied, for all $P,Q,R\in \operatorname{Mat}_N(\CC)[x]$, $a,b\in \CC$, $T\in \operatorname{Mat}_N(\CC)$
\begin{enumerate}
  \item $\langle aP+bQ,R\rangle=a\langle P,R\rangle+ b\langle Q,R\rangle$,
  \item $\langle TP,R\rangle= T\langle P,R\rangle$,
  \item $\langle P,Q\rangle^*=\langle Q,P\rangle$,
  \item $\langle P,P\rangle\geq 0$. Moreover, if $\langle P,P\rangle=0$, then $P=0$.
\end{enumerate}

Let us denote $\NN_0=\NN\cup\{0\}$. Given a weight matrix $W$ one can construct sequences of matrix valued orthogonal polynomials, that is sequences $\{P_n\}_{n\in\NN_0}$, where $P_n$ is a polynomial of degree $n$ with nonsingular leading coefficient and $\langle P_n,P_m\rangle=0$ for $n\neq m$.
We observe that there exists a unique sequence of monic orthogonal polynomials $\{Q_n\}_{n\in\NN_0}$ in $\operatorname{Mat}_N(\CC)[x]$.
By following a standard argument, given for instance in \cite{K49} or \cite{K71}, one shows that the monic orthogonal polynomials $\{Q_n\}_{n\in\NN_0}$ satisfy a three-term recursion relation
$$x Q_n(x)=A_{n}Q_{n-1}(x)+ B_nQ_n(x)+ Q_{n+1}(x), \qquad n\in\NN_0,$$
where $Q_{-1}=0$ and $A_n, B_n$ are matrices depending on $n$ and not on $x$.

Two weights $W$ and $\widetilde W$ are said to be {\em similar} if there exists a nonsingular matrix $M$, which does not depend on $x$, such that
$$\widetilde W(x)=M W(x)M^*, \quad \text{ for all } x\in (a,b).$$
Notice that if $\{P_n\}_{n\geq 0}$ is a sequence of orthogonal polynomials with respect to $W$, and $M\in \mathrm{GL}_N(\CC)$, then $\{P_n M^{-1}\}_{n\geq 0}$ is orthogonal with respect to $\widetilde W=MWM^*$.
A weight matrix $W$ reduces  to a smaller size if there exists a nonsingular matrix $M$ such that
\begin{equation*}
 M W(x)  M^*= \begin{pmatrix}
    W_1(x) & 0\\ 0& W_2(x)
  \end{pmatrix}, \qquad \text{ for all } x\in (a,b),
\end{equation*}
where $W_1$ and $W_2$ are weights of smaller size.

For a given weight matrix and a sequence of  orthogonal polynomials, it may be of interest the study of the differential operators having these polynomials as eigenfunctions.
Let $D$ be a right-hand side ordinary differential operator with matrix polynomial coefficients $F_i(x)$ of degree less than or equal to $i$ of the form
\begin{equation}\label{D2}
  D=\sum_{i=0}^s \partial ^i F_i(x),\qquad \partial=\frac{d}{dx},
\end{equation}
with the action of $D$ on a polynomial function $P(x)$ given by
$$(PD)(x)=\sum_{i=0}^s \partial ^i (P)(x)F_i(x).$$
We say that  the differential operator $D$ is {\em symmetric} if $\langle PD,Q\rangle=\langle P,QD\rangle$, for all $P,Q\in \operatorname{Mat}_N(\CC)[x]$.
It is a matter of careful integration by parts to see that the condition of symmetry for a differential operator of order two  is equivalent to a set of  three differential equations involving the weight $W$ and the coefficients of the differential operator $D$.

\begin{prop}[\cite{GPT03} or \cite{DG04}] \label{equivDsymm}
Let $W(x)$ be a smooth weight matrix supported on $(a,b)$. Let $D=\partial ^2F_2(x)+\partial F_1(x)+F_0$.
Then $D$ is symmetric with respect to $W$ if and only if
\begin{equation*}\left\{\begin{aligned}
   F_2 W &=WF_2^*\\
    2(F_2W)'-F_1W &=WF_1^*\\
    (F_2W)''-(F_1W)'+F_0W&=WF_0^*
\end{aligned} \right. \end{equation*}
with the boundary conditions
$$\lim_{x\to a,b} F_2(x)W(x)=0, \quad \lim_{x\to a,b} \big (F_1(x)W(x)-WF_1^*(x)\big)=0.$$
\end{prop}

\

We consider the following subalgebra of the algebra of all right-hand side differential operators with coefficients in $\operatorname{Mat}_N(\CC)[x]$,
\begin{equation*}
  \mathcal D= \{D=\textstyle\sum_{i=0}^s \partial ^i F_i\, : \,s\in\NN_0,\, F_i\in \operatorname{Mat}_N(\CC)[x], \deg F_i\leq i\}.
\end{equation*}

\begin{prop}[\cite{GT07}, Propositions 2.6 and 2.7]\label{eigenvalue-prop}
  Let $W=W(x)$ be a weight matrix of size $N\times  N$ and let $\{Q_n\}_{n\geq 0}$ be  the sequence of monic orthogonal polynomials in $\operatorname{Mat}_N(\CC)[x]$. If $D$ is a right-hand side ordinary differential operator of order $s$, as in \eqref{D2}, such that
  $$Q_nD=\Lambda_n Q_n, \qquad \text{for all } n\in\NN_0,$$
  with $\Lambda_n\in \operatorname{Mat}_N(\CC)$, then
 $F_i=F_i(x)=\sum_{j=0}^i x^j F_j^i$, $F_j^i \in \operatorname{Mat}_N(\CC)$, is a polynomial and $\deg(F_i)\leq i$. Moreover $D$ is determined by the sequence $\{\Lambda_n\}_{n\geq 0}$ and
 \begin{equation}\label{eigenvaluemonicos}
   \Lambda_n=\sum_{i=0}^s [n]_i F_i^i, \qquad \text{for all } n\geq 0,
 \end{equation}
    where $[n]_i=n(n-1)\cdots (n-i+1)$, $[n]_0=1$.
\end{prop}

Given a matrix weight $W$, the algebra
\begin{equation}\label{algDW}
  \mathcal D(W)=\{D\in \mathcal D\, : \, P_nD=\Lambda_n(D) P_n, \, \Lambda_n(D)\in \operatorname{Mat}_N(\CC), \text{ for all }n\in\NN_0\}
\end{equation}
is introduced in \cite{GT07}, where $\{P_n\}_{n\in \NN_0}$ is any sequence of matrix valued orthogonal polynomials  with respect to $W$.

We observe that the definition of $\mathcal D(W)$ depends only on the weight matrix $W$ and not  on the particular sequence of orthogonal polynomials,
since two sequences $\{P_w\}_{w\in\NN_0}$ and $\{Q_w\}_{w\in\NN_0}$ of matrix orthogonal polynomials with respect to the weight $W$ are related by
$P_w=M_wQ_w$, for  ${w\in\NN_0}$, with $\{M_w\}_{w\in\NN_0}$ invertible matrices (see \cite[Corollary 2.5]{GT07}).

\begin{prop} [\cite{GT07}, Proposition 2.8]\label{prop2.8-GT}
For each $n\in\NN_0$, the mapping $D\mapsto \Lambda_n(D) $ is a representation of $\mathcal D(W)$ in $\operatorname{Mat}_N(\CC)$. Moreover, the sequence of representations $\{\Lambda_n\}_{n\in\NN_0}$ separates the elements of $\mathcal D(W)$ .
\end{prop}

We remark that the result in Proposition \ref{prop2.8-GT} says that the map $$D\mapsto ( \Lambda_0(D), \Lambda_1(D), \Lambda_2(D),\dots\dots)$$ is an injective morphism of $\mathcal D(W)$ into $\operatorname{Mat}_N(\CC)^{\NN_0}$, the direct product of infinite copies, indexed by $\NN_0$, of the algebra $\operatorname{Mat}_N(\CC)$.
In particular, if $D_1,D_2\in \mathcal D(W)$ then
\begin{equation*}
  D_1=D_2 \quad  \text{ if and only if } \quad  \Lambda_n(D_1)=\Lambda_n(D_2) \, \text{ for all }n\in\NN_0.
\end{equation*}

For any $D\in \mathcal D(W)$ there exists a unique differential operator $D^*\in \mathcal D(W)$, the adjoint of $D$ in $\mathcal D(W)$, such that
$$\langle PD,Q\rangle=\langle P,QD^*\rangle,$$ for all $P,Q\in \operatorname{Mat}_N(\CC)[x]$.
See Theorem 4.3 and Corollary 4.5 in \cite{GT07}.
The map $D\mapsto D^*$ is a *-operation in the algebra $\mathcal D(W)$. Moreover, it is shown that  $\mathcal S(W)$, the set of all symmetric operators in $\mathcal D(W)$, is a real form of the space $\mathcal D(W)$, i.e.
$$\mathcal D(W)= \mathcal S (W)\oplus i \mathcal S (W),$$
as real vector spaces. In particular, the algebra $\mathcal D(W)$, together with the involution, is completely determined by $\mathcal S(W)$.

\begin{cor}
A differential operator $D\in \mathcal D(W)$ is a symmetric operator if and only if   $$\Lambda_n(D)\langle Q_n, Q_n\rangle= \langle Q_n, Q_n\rangle \Lambda_n(D)^* $$ for all ${n\in\NN_0}$.
\end{cor}

Also it is worth to recall the following important result from \cite{GT07}.

\begin{prop}[Proposition 2.10]\label{symoperinD(W)}
  If $D\in \mathcal D$ is symmetric then $D\in \mathcal D(W)$.
\end{prop}

\section{Matrix valued orthogonal polynomials associated with the $n$-dimensional spheres}\label{sec-MOPn}

Motivated by the results obtained in \cite{TZ13B} we introduce the following family of polynomials, for $w\in\NN_0$,
\begin{equation} \label{Pwdef}
P_w(x)=P_w^{n,p}(x)=\begin{pmatrix}
\frac{1}{n+1}\, C_w^{\frac{n+1}{2}}(x)+\frac{1}{p+w}\,C_{w-2}^{\frac{n+3}{2}}(x)&\frac{1}{p+w}\,C_{w-1}^{\frac{n+3}{2}}(x)\\ \mbox{} \\
\frac{1}{n-p+w}\,C_{w-1}^{\frac{n+3}{2}}(x)&\frac{1}{n+1}\, C_w^{\frac{n+1}{2}}(x)+\frac{1}{n-p+w}\,C_{w-2}^{\frac{n+3}{2}}(x)
\end{pmatrix},
\end{equation}
with parameters $p,n\in \mathbb R$ such that $0< p< n$. Here   $C_n^\lambda(x)$ denotes the $n$-th Gegenbauer polynomial
$$C_w^\lambda(x)=\frac{(2\lambda)_w} {w!}\,
 {}_2\!F_1\left(\begin{matrix}-w,\,w+2\lambda\\ \lambda+1/2\end{matrix};\frac{1-x}{2}\right),  \qquad x\in[-1,1],
$$
 where $(a)_w=a(a+1)\dots (a+w-1)$ denotes the Pochhammer symbol. As usual, we assume $C_w^\lambda(x)=0$ if $w<0$.
We recall that $C_w^\lambda$ is a polynomial of degree $w$, with leading coefficient  $\frac{2^w(\lambda)_w}{w!}$.

\smallskip
Let us observe that  $\deg(P_w)=w$ and the leading coefficient of $P_w$ is a nonsingular scalar matrix
  \begin{equation}\label{leadcoef}
    \frac {2^w (\tfrac{n+1}2)_w}{(n+1)\,w!}\,\text{Id}=\frac 1{w!}{2^{w-1} (\tfrac{n+3}2)_{w-1}}\,\text{Id}.
  \end{equation}

\

We start by proving that the polynomials $P_w$ given in \eqref{Pwdef} are eigenfunctions of the following  differential operator $D$.

\begin{thm} \label{operatorD} For each $w\in \NN_0$, the matrix  polynomial $P_w$ is an eigenfunction of the differential operator
$$D= \partial^2 \,(1-x^2)-\partial \,\Big( (n+2)x+2\left(\begin{smallmatrix}
  0&1\\1&0\end{smallmatrix}\right) \Big)-\left(\begin{smallmatrix}
    p&0\\0&n-p  \end{smallmatrix}\right), $$
with eigenvalue
$$\Lambda_w(D)= \begin{pmatrix}
  -w(w+n+1)-p & 0\\ 0& -w(w+n+1)-n+p
\end{pmatrix}.$$
\end{thm}

\begin{proof} We need to verify that $$P_w D=\Lambda_wP_w.$$
We will need to use the following properties of the Gegenbauer polynomials (for the first three see \cite{KS} page 40, and for the last one see \cite{S75}, page 83, equation (4.7.27))
\begin{align}
\label{a0}& (1-x^2)\frac{d^2 }{dx^2} C_m^\lambda(x)-(2\lambda+1)x  \frac {d }{dx}C_{m}^{\lambda}(x) + m(m+2\lambda)C_{m}^{\lambda}(x)=0,\\
\label{aa}&\frac{d }{dx}C_m^\lambda(x)=2\lambda \,C_{m-1}^{\lambda+1}(x),\\
\label{bb}&2(m+\lambda)x\, C_m^\lambda (x)= (m+1)C_{m+1}^{\lambda}(x)+(m+2\lambda-1)C_{m-1}^{\lambda}(x),\\
\label{dd}&\frac{(m+2\lambda-1)}{2(\lambda-1)} C_{m+1}^{\lambda-1}(x)=  C_{m+1}^{\lambda}(x) -  x\, C_{m}^{\lambda}(x).
\end{align}

\noindent Also, combining  \eqref{bb} and \eqref{dd}, we have
\begin{equation}
  \label{new}
(m+\lambda) C_{m+1}^{\lambda-1}(x)=   (\lambda-1)  \Big (C_{m+1}^\lambda(x)-C_{m-1}^\lambda(x)\Big).
\end{equation}

\smallskip
The entry $(1,1)$ of the matrix $P_w D-\Lambda_wP_w$ is
\begin{align*}
\nonumber (1&-x^2)(P_w)''_{11}-(n+2)x(P_w)'_{11}-2(P_w)'_{12}+w(w+n+1) (P_w)_{11} \\
 & =(1-x^2)\left( \tfrac 1{n+1} \,C_w^{\tfrac{n+1}2}+\tfrac{1}{p+w} C_{w-2}^{\tfrac{n+3}2}\right)''-(n+2)x\left( \tfrac 1{n+1} \,C_w^{\tfrac{n+1}2}+\tfrac{1}{p+w} C_{w-2}^{\tfrac{n+3}2}\right)' \\
 \nonumber & \quad -\tfrac 2{p+w} \left( C_{w-1}^{\tfrac{n+3}2}\right)'+w(w+n+1) \left( \tfrac 1{n+1} \,C_w^{\tfrac{n+1}2}+\tfrac{1}{p+w} C_{w-2}^{\tfrac{n+3}2}\right).
\end{align*}

Applying \eqref{a0} for $\lambda=\frac12(n+1)$, $\lambda=\frac12(n+3)$ and \eqref{aa} for $\lambda=\frac12(n+3)$, with $m=w$, we have that the entry (1,1) of  $P_w D-\Lambda_wP_w$, multiplied by
$(p+w)/2$  is
\begin{align*}
-(n+3)C_{w-2}^{\tfrac{n+5}2}+(n+3)\,x\, C_{w-3}^{\tfrac{n+5}2}+ (w+n+1) C_{w-2}^{\tfrac{n+3}2} =0,
\end{align*}
this last identity follows from equation \eqref{dd} with $\lambda=\frac{n+5}2$ and $m=w-3$.
Repeating the previous verification, by changing $p$ by $n-p$, it follows that the entry $(2,2)$ of $P_w D-\Lambda_wP_w $ is also zero.

The entry $(1,2)$ of $P_w D-\Lambda_wP_w$ is
$$(1-x^2) (P_w)''_{12}-(n+2)x(P_w)'_{12}-2(P_w)'_{11}+\big(w(w+n+1)-n+2p\big) (P_w)_{12},$$
if we multiply it by $(p+w)$  we get
\begin{equation}\label{21entryD}
\begin{split}
 (1-x^2)&\Big(C_{w-1}^{\tfrac{n+3}2} \Big)'' -(n+2) x\Big(C_{w-1}^{\tfrac{n+3}2}\Big)' +( w(w+n+1)-n+2p) C_{w-1}^{\tfrac{n+3}2}
 - 2\tfrac{(p+w)}{n+1} \Big(C_{w}^{\tfrac{n+1}2}\Big)'-2\Big(C_{w-2}^{\tfrac{n+3}2}\Big)'.
\end{split}
\end{equation}

Applying \eqref{a0} for $\lambda=(n+3)/2$, $m=w-1$,  \eqref{aa} for $\lambda=(n+1)/2$, $m=w$ and $\lambda=(n+3)/2$, $m=w-1$, one obtain that   \eqref{21entryD}  is
\begin{equation*}\label{21entryDb}
  2x \Big(C_{w-1}^{\tfrac{n+3}2}\Big)'-2(w-1) C_{w-1}^{\tfrac{n+3}2}-2(n+3)C_{w-3}^{\tfrac{n+5}2}.
\end{equation*}
Now, applying \eqref{aa} and \eqref{dd}, this expression becomes
  \begin{equation*}\label{21entryDc}
2(n+3)\Big(C_{w-1}^{\tfrac{n+5}2}-C_{w-3}^{\tfrac{n+5}2}\Big) -2(2w+n+1) C_{w-1}^{\tfrac{n+3}2},
\end{equation*}
which is equal to zero by \eqref{new} with $\lambda=\frac{n+5}2$ and $m=w-2$. This concludes that the entry $(1,2)$ of $P_w D-\Lambda_wP_w$ is zero.
To complete the proof of the theorem we need to verify that the entry $(2,1)$  is also zero.
This is obtained making exactly the same computations, by changing $p$ by $n-p$.
\end{proof}

\smallskip
We introduce the weight matrix
\begin{equation}\label{peso-x}
  W(x)= W_{p,n}=(1-x^2)^{\tfrac n2 -1} \begin{pmatrix}
  p\,x^2+n-p & -nx\\ -nx & (n-p)x^2+p
\end{pmatrix},\quad x\in [-1,1].
\end{equation}

\begin{prop}
For $n\neq 2p$, the weight $W(x)$ does not reduce to a smaller size.
 \end{prop}
\begin{proof}
  Assume that there exists a nonsingular matrix $M=\left(\begin{matrix}
    m_{11}&m_{12}\\ m_{21}&m_{22}
  \end{matrix}\right)$ such that
$$MW(x)M^*=\begin{pmatrix}
    w_1(x) &0\\ 0& w_2(x)
  \end{pmatrix}.$$  The entry $(1,2)$ of $MW(x)M^*$ is
  $$ x^2\big( p\, m_{11}\overline m_{21}+ (n-p)m_{12}\overline m_{22} \big)-  \big(m_{11}\overline m_{22}+m_{12}\overline m_{21}\big) n\, x+ (n-p)m_{11}\overline m_{21}
  + p\, m_{12}\overline m_{22},$$
from here we see that
  \begin{align}
  m_{11}\overline m_{22}+m_{12}\overline m_{21}&=0, \nonumber\\
\label{aux5}    p\, m_{11}\overline m_{21}+ (n-p)m_{12}\overline m_{22}&=0, \\
\label{aux6}    (n-p) m_{11}\overline m_{21}+p \,m_{12}\overline m_{22}&=0.
  \end{align}
 By combining equations \eqref{aux5} and \eqref{aux6} we have that $(n-2p)m_{11}\overline m_{21}=0$.
 The assumption $n\neq2p$, together with \eqref{dd}, implies $\det(M)=0$, which is a contradiction.
\end{proof}
\begin{remark}\label{RR}
For $n=2p$, the weight matrix $W$ reduces to two scalar weights. The corresponding scalar polynomials are Jacobi polynomials $P_w^{\alpha, \beta}$ with $(\alpha,\beta)=(n/2+1,n/2-1)$ and $(\alpha,\beta)=(n/2-1,n/2+1)$, respectively.
In fact, by taking
$M=\left(\begin{smallmatrix}
  1&1\\-1&1
\end{smallmatrix}\right)$   we have that
\begin{align*}
M W(x) M^* & = 2p\, (1-x^2)^{\tfrac n2 -1}
 \begin{pmatrix}
  (1-x)^2 & 0\\ 0 & (1+x)^2
\end{pmatrix}.
\end{align*}
\end{remark}

\begin{remark} \label{p/n-p} We have that the weight matrices $W_{p,n}$ and $W_{n-p,n}$ are similar.
In fact, by taking $M=\left(\begin{smallmatrix}   0&1\\1&0 \end{smallmatrix}\right)$ we get
$$M W_{p,n}M^*  =W_{{n-p},n}.$$
\end{remark}

\smallskip
From Proposition \ref{equivDsymm} and following  straightforward computations, one can prove the  following result.
\begin{prop}\label{ref}
  The differential operator
  $$D= \partial^2 \,(1-x^2)-\partial \,\Big( (n+2)x+2\left(\begin{smallmatrix}
  0&1\\1&0\end{smallmatrix}\right) \Big)-\left(\begin{smallmatrix}
    p&0\\0&n-p  \end{smallmatrix}\right) $$
  is symmetric with respect to the weight function $W(x)$.
\end{prop}

\smallskip

In the scalar case, if $D$ is a symmetric differential operator with respect to  $W$ and $\{P_w\}_{w\in \NN_0}$ is a family of eigenfunctions of $D$ with different eigenvalues, then the sequence $\{P_w\}_{w\in \NN_0}$ is automatically orthogonal with respect to $W$. In the matrix case this is not always true since
\begin{equation} \label{sim}
\Lambda_w \langle P_w,P_{w'}\rangle=\langle P_w D, P_{w'}\rangle= \langle P_w, P_{w'} D\rangle= \langle P_w, P_{w'} \rangle\Lambda_{w'}
\end{equation}
does not imply that $\langle P_w,P_{w'}\rangle=0$, for $w\neq w'$. Therefore, we prove the orthogonality in the next theorem.

\begin{thm}\label{PwortogonalW}
When $n\ne 2p$  the matrix polynomials $\{P_w\}_{w\in\NN_0} $ are orthogonal polynomials with respect to the matrix valued inner product
  $$  \langle P,Q \rangle = \int_{-1}^1 P(x) W(x) Q(x)^*\,dx.$$
\end{thm}
\begin{proof}
We know that  $P_w$ is a polynomial of degree $w$ and its leading coefficient is a nonsingular diagonal matrix (see \eqref{leadcoef}).
We only have to verify that for $ w\neq w'$,
$\langle P_{w},P_{w'} \rangle _W  =0$.
Since $P_w$ is an eigenfunction of the differential operator $D$, which  is  symmetric with respect to $W$, we have that \eqref{sim} holds
with $$\Lambda_w=
\left(
\begin{smallmatrix}
\lambda_{w,1}&0\\0&\lambda_{w,2}
\end{smallmatrix}
\right)
=
\left(
\begin{smallmatrix}
-w(w+n+1)-p&0\\0&-w(w+n+1)-n+p
\end{smallmatrix}
\right),
$$
see Theorem \ref{operatorD}. Therefore,  for $i,j=1,2$ we have
$\lambda_{w,i} \langle P_{w,i},P_{w', j} \rangle
=\lambda_{w',j} \langle P_{w,i},P_{w', j} \rangle,$
 where $P_{w,i}$ is the $i$-th row of the  polynomial $P_w$,
and
\begin{equation*}
 \langle P_{w,i},P_{w', j} \rangle
 =\int_{-1}^1  P_{w,i}(x)\,W(x) P_{w',j}^*(x) \, dx \in \CC.
\end{equation*}

It is not difficult to verify that $\lambda_{w,i}\neq\lambda_{w',j}$, for  $w\neq w'$  or $i\neq j$.
Then
we have
\begin{equation} \label{normaPwdiagonal}
 \langle P_{w,i},P_{w', j} \rangle=0,
 \quad \text{ for } w\neq w' \text{ or } i\neq j.
\end{equation}
Therefore $\langle P_{w},P_{w'} \rangle   =0$, for $w\neq w'$, which concludes the proof of the theorem.
\end{proof}

\begin{remark}\label{comparacionK}
Recently, in \cite{KRR14} the authors study some families on matrix valued polynomials, depending on one real parameter $\nu >0 $, of arbitrary size $(2\ell+1)\times (2\ell+1)$ with $\ell\in \frac 12 \NN$. These weights are not irreducible. For  $\ell=1, \frac 32, 2$  appears some  irreducible  $2\times 2$ blocks $W_+^{(\nu)}$ and $W_-^{(\nu)}$. See Remark 2.8 (ii) there.

The case $\ell=3/2$ does not match with the examples considered in this paper. The cases $\ell=1$ and $\ell=2$ are particular cases of our weight matrices $W_{p,n}$ by choosing  our parameters $(p,n)=(\nu, 2\nu+1)$ and   $(p,n)=(\nu, 2\nu+3)$, for  $\ell=1$ and $\ell=2$ respectively.
In fact, with
$L=\left(\begin{smallmatrix}
  0& \sqrt 2 \\ -1 & 0
\end{smallmatrix}\right)$ and $D=\left(\begin{smallmatrix}
  1& 0\\0 & -2
\end{smallmatrix}\right)$ we get
\begin{align*}
W_+^{(\nu)}&=  \frac{(\nu+2)}{(2\nu+1)}L\, W_{\nu, 2\nu+1 } L^* \quad  \qquad  \qquad \text{ for }\ell=1,\\
W_-^{(\nu)}&= \frac{(\nu+4)(\nu+2)}{(2\nu+1)(2\nu+3)} D\, W_{\nu, 2\nu+3 } D^* \quad  \text{ for }\ell=2.
\end{align*}
The case $\nu=1$ was  previously studied in \cite{KPR12}  and \cite{KPR13}.
\end{remark}

\section{Three-term recursion relation}\label{sec-ttrr}

The main result of this section is a three-term recursion relation satisfied by the sequence of orthogonal polynomials studied in this paper. We give a  proof by using some properties of the Gegenbauer polynomials.

\begin{thm}\label{ttrrthm}
The orthogonal polynomials $\{P_w\}_{w\in\NN_0}$ satisfy the three-term recursion relation
$$ x\,P_w(x)= A_w P_{w-1}(x)+B_wP_w(x)+C_w P_{w+1}(x), $$
where
\begin{align*}
  A_w& = \begin{pmatrix}
  \tfrac{(n+w)(p+w-1)(n-p+w+1)}{(p+w)(n-p+w)(2w+n+1)} &0\\0&  \tfrac{(n+w)(p+w+1)(n-p+w-1)}{(p+w)(n-p+w)(2w+n+1)}
\end{pmatrix},\\ \mbox{} \\
 B_w &=\begin{pmatrix}
    0& \tfrac {-p}{(p+w)(p+w+1)}\\ \tfrac {-(n-p)}{(n-p+w)(n-p+w+1)}& 0
  \end{pmatrix}\, , \qquad C_w=\tfrac{w+1}{2w+n+1}I.
\end{align*}
\end{thm}

\begin{proof}

To verify the $(1,1)$-entry of the equation in the statement of the theorem  we need to prove that
\begin{equation}\label{aux1}
\begin{split}
  x \Big( \tfrac 1{n+1}& C_w^{{\frac{n-1}2}-1}(x)+ \tfrac 1{p+w}C^{\frac{n+3}2}_{w-2}(x) \Big)
  =\tfrac{(n+w)(p+w-1)(n-p+w+1)}{(2w+n+1)(p+w)(n-p+w)}\left(\tfrac 1{n+1}C_{w-1}^{{\frac{n-1}2}-1}(x)+\tfrac 1{p+w}C^{\frac{n+3}2}_{w-3}(x)\right)\\
  &\quad -\tfrac p{(p+w)(p+w+1)(n-p+w)}C_{w-1}^{\frac{n+3}2}(x)
   + \tfrac{w+1}{2w+n+1} \left(\tfrac 1{n+1}C_{w+1}^{{\frac{n-1}2}-1}(x)+\tfrac 1{p+w-1}C^{\frac{n+3}2}_{w-1}(x)\right).
\end{split}
\end{equation}

By replacing the identities given by \eqref{bb} for $\lambda=\frac{n+1}2$, $m=w$ and $\lambda=\frac{n+3}2$, $m=w-2$,
one obtain that \eqref{aux1} is equivalent to
\begin{equation}\label{aux2}
\begin{split}
  &\tfrac{(w+n)}{(n+1)(2w+n+1)} \left(  -1+ \tfrac {(p+w-1)(n-p+w-1)} {(p+w)(n-p+w)} \right) C_{w-1}^{{\frac{n+1}2}}(x)\\
& \quad   + \left(  -\tfrac{p}{(p+w)(p+w+1)(n-p+w)}+ \tfrac{w+1}{(2w+n+1)(p+w+1)}-\tfrac{w-1}{(p+w)(2w+n-1)}\right) C_{w-1}^{\frac{n+3}2}(x)\\
& \quad  +\tfrac{(n+w)}{p+w} \left( (\tfrac{n-p+w-1}{(2w+n+1)(n-p+w)}-\tfrac 1{2w+n-1} \right)C_{w-3}^{{\frac{n+3}2}}(x)=0.
\end{split}
\end{equation}

\noindent Thus, by using the  relation \eqref{new} for $\lambda=\frac{n+3}{2}$ and $m=w-2$,
the identity in \eqref{aux2} follows after some straightforward computations.

\smallskip

Now we verify that the equation for the $(1,2)$-entry in the statement of the theorem holds. We need to verify that the following identity holds
\begin{equation}\label{aux3}
  \begin{split}
    \tfrac 1{p+w}&\,x\, C_{w-1}^{\frac{n+3}2}(x)= \tfrac{(n+w)(n-p+w+1)}{(p+w)(2w+n+1)(n-p+w)}C_{w-2}^{\frac{n+3}2}(x) \\ &- \tfrac{p}{(p+w)(p+w+1)}\left( \tfrac 1{n+1}C_w^{{\frac{n+1}2}}(x) +\tfrac 1{n-p+w} C_{w-2}^{\frac{n+3}2} (x)   \right)
   + \tfrac{w+1}{(2w+n+1)(p+w+1)} C_w^{\frac{n+3}2}(x).
  \end{split}
\end{equation}

\noindent From \eqref{new} for $\lambda=\frac{n+3}{2}$ and  $m=w-1$ we have that the right-hand side of \eqref{aux3} is
\begin{align*}
     \tfrac{n+w+1}{(p+w)(2w+n+1)} C_{w-2}^{\frac{n+3}2}(x) + \tfrac{w}{(p+w)(2w+n+1)}C_{w}^{\frac{n+3}2}(x).
  \end{align*}
Therefore, \eqref{aux3} is proved, since it is equivalent to \eqref{bb} with $\lambda=\tfrac{n+3}2$ and $m=w-1$.

For  the entries $(2,2)$ and $(2,1)$ we proceed in a similar way, by observing that we need to do the same computations as in the cases $(1,1)$ and $(1,2)$ respectively,  by changing $p$ by $n-p$.
This concludes the proof of the theorem.
\end{proof}

The  sequence  of monic orthogonal polynomials is given by
\begin{equation}\label{monicos}
Q_w=   \frac {w!(n+1)}{2^{w} \left(\tfrac{n+1}2\right)_{w}}  P_w, \qquad {w\in\NN_0}.
\end{equation}
 The first polynomials of the sequence $\{Q_w\}_{w\in \NN_0}$ are
\begin{align*}Q_0&=\operatorname{Id},
\qquad Q_1=\begin{pmatrix}
  x& \tfrac 1{p+1}\\ \tfrac 1{n-p+1}& x
\end{pmatrix},  \qquad
Q_2 =\begin{pmatrix}
  x^2-\tfrac p{(n+3)(p+2)} & \frac {2}{p+2}x \\ \mbox{} \displaybreak[0]\\ \frac{2}{n-p+2}x & x^2-\tfrac{n-p}{(n+3)(n-p+2)}
\end{pmatrix}, \\
Q_3&=\begin{pmatrix}
x^3-\frac{3(p+1)}{(n+5)(p+3)}  x & \frac{3}{p+3}x^2-\frac{3}{(n+5)(p+3)}\\ \mbox{} \displaybreak[0]\\
\frac{3}{n-p+3}x^2-\frac{3}{(n+5)(n-p+3) }& x^3-\frac{3(n-p+1)}{(n+5)(n-p+3)}  x
\end{pmatrix}.
\end{align*}

\smallskip
\begin{remark}
  Observe that from \eqref{normaPwdiagonal} and \eqref{monicos} we have that  $\langle Q_w,Q_w\rangle$ is always a diagonal matrix. Moreover one can verify that
\begin{align*}
 \langle Q_w,Q_w\rangle  =\| Q_w\|^2=
\frac{\pi2^{[w/2]}\Gamma(n/2+1+[w/2]) }{w!(n+2w+1)\Gamma((n+3)/2)}\prod_{k=1}^{[(w-1)/2]}(n+2k+1)
\begin{pmatrix}
    \frac{p\,(n-p+w+1)}{p+w} &0\\0& \frac{(n-p)(p+w+1)}{n-p+w}
  \end{pmatrix}.
\end{align*}
\end{remark}

\section{The algebra $\mathcal D(W)$}\label{sec-dw}

In this section we discuss some properties of the structure of the algebra $\mathcal D(W)$, defined  in \eqref{algDW}, for our weight matrix $W(x)$
introduced in \eqref{peso-x}. We are not interested in the cases when $p=n-p$, since the weight reduces to classical scalar weights, see Remark \ref{RR}.
We observe that  in our example, the polynomials $\{P_w\}_{w\in \NN_0}$, given in \eqref{Pwdef},
and the monic orthogonal polynomials $\{Q_w\}_{w\in \NN_0}$
have the same sequence of eigenvalues, since they are related by a scalar multiple, see \eqref{monicos}.

\smallskip
First of all we observe that the space of differential operators of {\em order zero } in $\mathcal D(W)$ consists of scalar multiplies of the identity operator.
In fact, a differential operator of order zero, having the sequence of monic orthogonal polynomials $\{Q_w\}_w$ as eigenfunctions,
is a constant matrix $L$  such that
$$Q_wL=\Lambda_w\,Q_w,\qquad \text{ for all } {w\in\NN_0}.$$
From \eqref{eigenvaluemonicos} we have that $\Lambda_w=L$ for every $w$. When $w=1$, we obtain that the entries of $L$ satisfy $L_{11}=L_{22}$ and $(p+1)L_{12}=(n-p+1)L_{21}$. Thus, looking at the case $w=2$ we get $(n-2p)L_{12}=0$.
Therefore we obtain that any operator of order zero $L$ in $\mathcal D(W)$  is a multiple of the identity matrix.

\medskip
Now we study differential operators of  order at most two in the algebra $\mathcal D(W)$.
Let $\{Q_w\}_{w\in \NN_0}$ the  sequence of monic orthogonal polynomials with respect to $W$ and $D$ a differential operator of order at most two in $\mathcal D(W)$. From Proposition \ref{eigenvalue-prop} we have
\begin{equation*}
  D=\partial^2 (A_2 x^2+A_1 x+A_0) + \partial (B_1x+B_0)+C \,\in \mathcal D(W)
\end{equation*}
if and only if
\begin{equation*}
Q_wD=\big(w(w-1)A_2+wB_1+C\big)Q_w, \qquad \text{ for all ${w\in\NN_0}$}.
\end{equation*}
Here $A_2, A_1, A_0, B_1, B_0, C$ are $2\times 2$ complex matrices. Let us denote $Q_{w,j}$ the coefficients of the polynomial $Q_w$, i.e.,
$Q_w=\sum_{j=0}^w Q_{w,j}\,x^j$, with $Q_{w,w}=I$. Therefore
$D\in \mathcal D(W)$ if and only if
\begin{equation*}
  \begin{split}
    j(j-1)&Q_{w,j}A_2+j(j+1)Q_{w,j+1}A_1+(j+1)(j+2)Q_{w,j+2}A_0+jQ_{w,j}B_1\\
    & +(j+1)Q_{w,j+1}B_0+Q_{w,j}C-\big(w(w-1)A_2+wB_1+C\big)Q_{w,j}=0
  \end{split}
\end{equation*}
for all ${w\in\NN_0}$ and $j=0,\dots , w$.
For $j=w-1$  and $j=0$ we respectively obtain
\begin{equation}\label{eqj=w-1}
  \begin{split}
    (w-&1)(w-2)Q_{w,w-1}A_2+w(w-1)A_1+(w-1)Q_{w,w-1}B_1+ w B_0
    +Q_{w,w-1}C \\
    &-\big(w(w-1)A_2+wB_1 +C\big)Q_{w,w-1}=0
  \end{split}
\end{equation}
and
\begin{equation}\label{eqj=0}
  2Q_{w,2}A_0+Q_{w,1}B_0+Q_{w,0}\,C-\big(w(w-1)A_2+wB_1+C\big)Q_{w,0}=0.
\end{equation}

Now from \eqref{eqj=w-1} considering $w=1$ and $w=2$, and from \eqref{eqj=0} considering $w=2$, we respectively obtain
\begin{align*}
  B_0&=(B_1+C)Q_{1,0}-Q_{1,0}C,\qquad
  2A_1=(2A_2+2B_1+C)Q_{2,1}-Q_{2,1}B_1-2B_0-Q_{2,1}C,\\
  2A_0&=(2A_2+2B_1+C)Q_{2,0}-Q_{2,1}B_0-Q_{2,0}C.
\end{align*}

\noindent From the  expression of $Q_1$ and $Q_2$, given at the end of Section \ref{sec-ttrr}, we know that
\begin{align*}
 Q_{1,0}=
\begin{pmatrix}
  0& \tfrac 1{p+1}\\ \tfrac 1{n-p+1}&
\end{pmatrix},&&
Q_{2,1} =\begin{pmatrix}
 0 & \frac {2}{p+2} \\ \frac{2}{n-p+2} &  0
\end{pmatrix},&&
Q_{2,0}= \tfrac{-p}{(n+3)}\begin{pmatrix}
  \tfrac 1{(p+2)} &0\\0 &  \tfrac{1}{(n-p+2)}
\end{pmatrix}.
\end{align*}

\noindent By using \eqref{monicos} and \eqref{Pwdef} it is easy to see that
$$Q_{w,w-1}=\begin{pmatrix}
  0& \tfrac w{p+w}\\ \tfrac w{n-p+w}
\end{pmatrix}, \qquad \text{ for all } {w\in\NN}.$$

To determine the matrices $A_2=(a_{ij})$, $B_1=(b_{ij})$ and $C=(c_{ij})$, we first combine the entries in the diagonal of the matrix \eqref{eqj=w-1} to obtain
\begin{align*}
  2(n+2)a_{21}&=\frac {\big((n+p+2) b_{21}-2c_{21}\big)} {p+1} + \frac{(p+2)(p+w)(2c_{12}-(n-p)b_{12})}{(n-p+1)(n-p+2)(n-p+w)}, \\
  2(n+2)a_{12}&=\frac {\big((2n-p+2) b_{12}-2c_{12}\big)} {n-p+1} + \frac{(n-p+2)(n-p+w)(2c_{21}-p\,b_{21})}{(p+1)(p+2)(p+w)}.
\end{align*}

Since these identities are valid for any integer $w\geq 3$ we conclude that, if $n\neq 2p$ then
 $ 2c_{12}=(n-p)b_{12}$ and $2c_{21}=p\, b_{21}.$
Therefore
$b_{21}=2(p+1)a_{21}$
and $b_{12}=2(n-p+1)a_{12}$.

By combining the nondiagonal entries  of \eqref{eqj=w-1} we have
\begin{align*}
  (n-2p+1)\big( (n+2)a_{11}-b_{11}\big)= (n-2p-1)\big( (n+2)a_{22}-b_{22}\big)
\end{align*}
and
$$c_{11}-c_{22}= (p+1)(p+2)a_{22}-p(p+1)a_{11}+p\,b_{11}-(p+1)b_{22}.$$

Equation \eqref{eqj=0} with $w=3$ says that
$$ 2Q_{3,2}A_0+Q_{3,1}B_0+Q_{3,0}\,C-\big(6A_2+3B_1+C\big)Q_{3,0}=0.$$
Now, by using the expression of $Q_3=x^3+Q_{3,2}x^2+Q_{3,1}x+Q_{3,0}$ given at the end of Section \ref{sec-ttrr}, it is not difficult to see that
 $ b_{11}=(n+2)a_{11}$.
Thus
  $  b_{22}=(n+2)a_{22}$,
  and $  c_{11}-c_{22}=p(n-p+1)a_{11}-(p+1)(n-p)a_{22}$.

Therefore, the matrices $A_2 ,A_1 ,A_0,B_1,B_0,C$ are given in terms of the entries of $A_2$ and $c_{11}$, as we state in the following theorem.

\begin{thm} \label{order2op}
  The differential operators  of order at most two in $\mathcal D(W)$ are of the form
$$ D=\partial^2 F_2(x)+\partial F_1(x)+F_0,$$
where
\begin{equation} \label{coeffD}
   \begin{split}
  F_2(x) =& x^2\begin{pmatrix} a_{11}& a_{12}\\ a_{21}&a_{22} \end{pmatrix}
+ x \begin{pmatrix}  a_{12}-a_{21}& a_{11}-a_{22}\\ a_{22}-a_{11} & a_{21}-a_{12} \end{pmatrix}
+ \begin{pmatrix}   a_{22}& a_{21}\\ a_{12}&a_{11} \end{pmatrix} ,
\\
F_1(x)=&x\!\begin{pmatrix}   (n+2) a_{11}\!& 2(n-p+1)a_{12}\\2(p+1)a_{21}\!& (n+2)a_{22} \end{pmatrix}
 +\begin{pmatrix} -pa_{21}+ (n-p+2)a_{12}& (n-p+2)a_{11}-(n-p)a_{22}\\ -p a_{11}+(p+2)a_{22 }& (p+2)a_{21}-(n-p)a_{12}
\end{pmatrix},
\\
F_0=& \begin{pmatrix} p\,(n-p+1)a_{11} +c& (n-p)(n-p+1)a_{12}\\ p\,(p+1)a_{21} & (p+1)(n-p)a_{22}+c
\end{pmatrix}.
   \end{split}
 \end{equation}
with $a_{11}$, $a_{12}$, $a_{21}$, $a_{22}$, $c$ arbitrary complex numbers.
\end{thm}
\begin{proof}
We have already proved that any differential operator of order at most two in $\mathcal D(W)$ is of this form for some constant $a_{11}, a_{12}, a_{21}, a_{22}, c\in \CC$.
Let  $\mathcal D_2$ be the complex  vector space of the differential operators in $\mathcal D(W)$ of order at most two. Then we have that
$\operatorname {dim} \mathcal D_2\leq 5.$

From Proposition \ref{equivDsymm} it is not difficult to see that a differential operator $D$ of order two,  with coefficients given by  \eqref{coeffD}, is a symmetric operator if and only if
\begin{equation*}\label{opsymm2}
  a_{11}, a_{22}, c\in \RR \qquad \text{and} \quad p \,a_{21}= (n-p)\,\overline a_{12}.
\end{equation*}

From Proposition \ref{symoperinD(W)} any symmetric operator $D\in \mathcal D$ belongs to the algebra $\mathcal D(W)$.
Thus there exists (at least) five $\RR$-linearly independent symmetric operators in $\mathcal D_2$. Therefore
$\dim \mathcal D_2=5$ and
this concludes the proof of the theorem.
\end{proof}

\begin{cor}
  There are no operators of order one in the algebra $\mathcal D(W)$.
\end{cor}

The elements of the sequence $\{Q_w\}_w$ are eigenfunctions of the operators  $D\in \mathcal D(W)$ and they satisfy
$Q_wD=\Lambda_w(D)Q_w$ ,for  $ {w\in\NN_0}$.
We explicitly state  the eigenvalues $\Lambda_w$ using formula \eqref{eigenvaluemonicos}: for a differential operator
$D= \partial^2 F_2+\partial F_1+F_0$ we have
$$\Lambda_w(D)=w(w-1)F_2^2+wF_1^1+F_0^0,$$
with $F_i^i$ (i=1,2,3) the leading coefficient of the polynomial coefficient $F_i$ of the differential operator $D$.
Therefore we get

\begin{cor} Let $D\in \mathcal D(W)$, defined as in Theorem \ref{order2op}.
 The  monic orthogonal polynomials $\{Q_w\}_w$ satisfy $$Q_w D=\Lambda_w(D) Q_w, \qquad \text{ for } {w\in\NN_0},$$
where the eigenvalue $\Lambda_w(D)$ is given by
$$\Lambda_w(D)=\begin{pmatrix}
  (w+p)(w+n-p+1)a_{11}+c& (w+n-p)(w+n-p+1)a_{12}\\(w+p)(w+p+1)a_{21}&(w+n-p)(w+p+1)a_{22}+c
\end{pmatrix}.$$
\end{cor}

\smallskip
Now we introduce a useful basis for the differential operators of order at most two in the algebra $\mathcal D(W)$: the identity $I$ and
\begin{align*}
  D_1&= \partial ^2 \begin{pmatrix}  x^2& x\\-x&-1 \end{pmatrix}
     + \partial \begin{pmatrix} (n+2)x&n-p+2 \\-p &0\end{pmatrix}
     +\begin{pmatrix} p\,(n-p+1)&0\\0& 0  \end{pmatrix},
     \displaybreak[0]\\
  D_2&=\partial^2\begin{pmatrix}  -1& -x\\x&x^2 \end{pmatrix}
     + \partial \begin{pmatrix} 0&p-n\\p+2 &(n+2)x\end{pmatrix}
    +\begin{pmatrix}  0&0\\0& (p+1)(n-p) \end{pmatrix}  ,
    \displaybreak[0]\\
  D_3&=\partial^2 \begin{pmatrix}  -x& -1\\ x^2&x \end{pmatrix}
     +\partial \begin{pmatrix} -p& 0\\2(p+1)x &p+2\end{pmatrix}
    +\begin{pmatrix}  0&0\\p(p+1) & 0 \end{pmatrix} ,
    \displaybreak[0]\\
  D_4&=\partial^2\begin{pmatrix}  x&x^2 \\-1&-x \end{pmatrix}
     + \partial \begin{pmatrix} n-p+2 & 2(n-p+1) x\\0 &p-n\end{pmatrix}
    +\begin{pmatrix}  0&  (n-p)(n-p+1)\\0& 0 \end{pmatrix}.
\end{align*}

\smallskip
\noindent The corresponding eigenvalues are
\begin{align*}
  \Lambda_w(D_1)&=
  \left(\begin{smallmatrix}
  (w+p)(w+n-p+1) & 0\\ 0& 0
\end{smallmatrix}\right), \,
&&\Lambda_w(D_2) = \left(\begin{smallmatrix}
 0 & 0\\ 0& (w+p+1)(w+n-p)
\end{smallmatrix}\right) ,\\
\Lambda_w(D_3)& = \left(\begin{smallmatrix}
 0 & 0\\ (w+p)(w+p+1) & 0
\end{smallmatrix}\right), \,
&&\Lambda_w(D_4) = \left(\begin{smallmatrix}
 0 & (w+n-p)(w+n-p+1)\\ 0& 0
\end{smallmatrix}\right).
\end{align*}

\begin{remark}
  The differential operator $D$ appearing  in Theorem \ref{operatorD} is
  $D=-D_1-D_2+p(n-p)I.$
\end{remark}

We observe here that, for example,
$$ \Lambda_w(D_1) \Lambda_w(D_3) \neq \Lambda_w(D_3) \Lambda_w(D_1), \quad \text{ for all } {w\in\NN_0}.$$
By using Proposition \ref{prop2.8-GT} we obtain that $D_1D_3\neq D_3D_1$, which in turn implies the following result.

\begin{cor}
  The algebra $\mathcal D(W)$ is not commutative.
\end{cor}

By following the same argument, through the sequence of eigenvalues,  we obtain the following relations among the differential operators $D_1, D_2,D_3,D_4$.
\begin{align*}
D_1D_2& =0,   \quad D_2D_1=0, \quad    D_1D_3=0,     \quad D_4D_1=0,  \quad
D_2D_4=0,   \quad D_3D_2=0,   \quad  D_3^2=0,   \quad   D_4^2 =0, \\
D_3D_1&= D_2D_3-(n-2p)D_3, \quad D_1D_4= D_4D_2-(n-2p)D_4,  \quad
D_3D_4= D_2^2-(n-2p)D_2,\\ D_4D_3&=D_1^2+(n-2p)D_1.
\end{align*}

\begin{conj} \mbox{}
\begin{enumerate}
\item    There are no operators of odd order in $\mathcal D(W)$.
\item  The second order differential operators in $\mathcal D(W)$ generate  the algebra $\mathcal D(W)$.
\end{enumerate}
\end{conj}

\smallskip
 For a differential operator of order two $D=\partial ^2 F_2+\partial F_1 +F_0 \in \mathcal D(W)$, the explicit expression of  the adjoint operator $D^*$ is
\begin{equation*}
  D^*=\partial ^2 G_2+\partial G_1 +G_0,
\end{equation*}
where the polynomials $G_i$, $i=0,1,2$, are defined by
\begin{align*}
  G_0&= \langle Q_0, Q_0\rangle \Lambda_0(D)^* \langle Q_0, Q_0\rangle^{-1},\qquad
  G_1= \langle Q_1, Q_1\rangle \Lambda_1(D)^* \langle Q_1, Q_1\rangle^{-1} Q_1(x)-Q_1(x) G_0,\\
  G_2&= \langle Q_2, Q_2\rangle \Lambda_2(D)^* \langle Q_2, Q_2\rangle^{-1} Q_2(x)-\partial(Q_2)G_1(x)-Q_2(x) G_0,
\end{align*}
see Theorem 4.3 in \cite{GT07}.

\smallskip

Also from Corollary 4.5 in \cite{GT07}, we obtain the expression for the corresponding eigenvalues for the adjoint operator $D^*$, in terms of the eigenvalues of the differential operator $D$ and the norm of the polynomials $Q_w$,
$$\Lambda_w(D^*)= \langle Q_w, Q_w\rangle \Lambda_w(D)^* \langle Q_w, Q_w\rangle^{-1},\quad \text{for all }w.$$

By using the expressions of $\langle Q_i, Q_i\rangle$, given at the end of Section \ref{sec-ttrr}, and making straightforward computations, we can verify that
$$D_1^*=D_1,\quad D_2^*=D_2,\quad \text{and}\quad D_3^*=\tfrac{p}{n-p} D_4.$$
Therefore
$$ E_3=(n-p)D_3+pD_4\qquad \text{and} \qquad E_4=i\big((n-p)D_3-pD_4\big)$$
are also symmetric operators, because for any $D\in \mathcal D(W)$ the operators $D+D^*$ and $i(D-D^*)$ are symmetric operators.
 Explicitly,
\begin{align*}
  E_3=(n-p)D_3+pD_4 &
  =\partial^2 \begin{pmatrix}  -x(n-2p) & x^2 p-n+p   \\  x^2(n-p)-p&x(n-2p) \end{pmatrix}
     +\partial \begin{pmatrix} 2p& 2p(n-p+1)x\\ 2(p+1)(n-p) x&2(n-p) \end{pmatrix} \\ &\quad
    +\begin{pmatrix}  0&p(n-p)(n-p+1)\\p(p+1)(n-p) & 0 \end{pmatrix}, \displaybreak[0]
  \end{align*}
  \begin{align*}
 -i E_4 =(n-p)D_3-pD_4 &
  =\partial^2 \begin{pmatrix}  -n x & -x^2 p-n+p   \\  x^2(n-p)+p& n x \end{pmatrix}
     +\partial \begin{pmatrix} -2p(n-p+1)& -2p(n-p+1)x\\ 2(p+1)(n-p) x&2(n-p)(p+1) \end{pmatrix} \\ &\quad
    +\begin{pmatrix}  0&-p(n-p)(n-p+1)\\p(p+1)(n-p) & 0 \end{pmatrix}.
\end{align*}

\noindent The corresponding eigenvalues are
\begin{align*}
\Lambda_w\big (E_3\big)=&
\begin{pmatrix}
0& p(n-p+w)   (n-p+w+1) \\ (n-p)(p+w)(p+w+1) &0
\end{pmatrix},
\\
\Lambda_w\big (-iE_4\big)=&
\begin{pmatrix}
0& \!\!-p(n-p+w)   (n-p+w+1) \\ (n-p)(p+w)(p+w+1) &0
\end{pmatrix}.
\end{align*}

\begin{remark}
In \cite{KPR12} the authors study  matrix valued  orthogonal polynomials related to spherical functions on the group $(\mathrm{SU}(2)\times \mathrm{SU}(2), \mathrm{SU}(2))$.
The weight matrix is $W_+^{(\nu)}$, with $\nu=1$ in the notation of Remark \ref{comparacionK}.
Let us denote $\widetilde D_1$, $\widetilde D_2$ and $\widetilde D_3$ the differential operators $D_1$,$D_2$ and $D_3$ appearing in Theorem 8.1 in
\cite{KPR12}.
Then we have the following relations with our operators $D_1$, $D_2$, $D_3$ and $D_4$ for the case $n=3$ and $p=1$
$$\widetilde D_1 =L(D_1+D_2-3)L^{-1}, \quad \widetilde D_2 =L D_2 L^{-1}, \quad \widetilde D_3 =-\sqrt 2 \,L (2D_3+D_4)L^{-1}.$$
\end{remark}

\medskip
\noindent {\bf Acknowledgements.}  We would like to thank the referees for many useful comments
and suggestions that helped us to improve a first version of this paper.

\end{document}